\documentclass[11pt]{article}
\usepackage[english,activeacute]{babel}
\usepackage{amsmath,amsfonts,amssymb,amstext,amsthm,amscd,mathrsfs,amsbsy}
\usepackage{graphics,graphicx}
\usepackage{xypic}

\newtheorem{teo}{Theorem}[section]
\newtheorem{pro}[teo]{Proposition}
\newtheorem{coro}[teo]{Corollary}
\newtheorem{lem}[teo]{Lemma}

\theoremstyle{definition}
\newtheorem{defi}[teo]{Definition}
\newtheorem{exam}[teo]{Example}
\newtheorem{rem}[teo]{Remark}

\newcommand{\T}{\mathbb T}

\newcommand{\N}{\mathbb N}
\newcommand{\Z}{\mathbb Z}
\newcommand{\Q}{\mathbb Q}
\newcommand{\R}{\mathbb R}
\newcommand{\C}{\mathbb C}

\newcommand{\Sp}{\mbox{Spec }}
\newcommand{\m}{\mathfrak m}

\newcommand{\Su}{k[x^{a_1},\ldots,x^{a_s}]}
\newcommand{\Po}{k[x_1,\ldots,x_d]}
\newcommand{\Na}{Nash_n(X)}
\newcommand{\Nn}{\overline{Nash_n(X)}}

\newcommand{\GF}{GF(J_n)}

\newcommand{\sd}{\check{\sigma}}
\newcommand{\Si}{\mbox{Sing}}

\begin{document}

\date{}

\author{Daniel Duarte\footnote{Research supported by CONACYT (M\'{e}xico).}\\
Institut de Math\'ematiques de Toulouse\\
118 route de Narbonne, F-31062 Toulouse Cedex 9\\
dduarte@math.univ-toulouse.fr}
\title{Higher Nash blowup on normal toric varieties}

\maketitle

\begin{abstract}
The higher Nash blowup of an algebraic variety replaces singular points with limits of
certain spaces carrying higher order data associated to the variety at non-singular points.
In the case of normal toric varieties we give a combinatorial
description of the higher Nash blowup in terms of a Gr\"obner fan.
This description will allow us to prove the analogue of Nobile's theorem on the usual
Nash blowup in this context. More precisely, we prove that for a normal toric variety,
the higher Nash blowup is an isomorphism if and only if the variety is non-singular.
\end{abstract}


%

\section*{Introduction}

The classical Nash blowup is a natural modification of an algebraic variety that replaces singular
points by limits of tangent spaces at non-singular points. Recently, Takehiko
Yasuda in \cite{Y} has generalized this construction by considering not only first-order data, as
with the tangent space, but also higher-order one. In his construction, instead of tangent spaces,
the author considers nth infinitesimal neighborhoods of non-singular points.
Then one replaces singular points by limits of these infinitesimal neighborhoods at non-singular points.
The resulting variety is called higher Nash blowup of order $n$ and is denoted by $\Na$. Yasuda then
conjectures that for $n\gg0$, $\Na$ is non-singular (\cite{Y}, Conjecture 0.2). If the conjecture is true,
this process would give resolution of singularities in one step.
\\
\\
In this paper we consider Yasuda's higher Nash blowup in the case of normal toric varieties. Let
$\sigma\subset\R^d$ be a strictly convex rational polyhedral cone, $X$ the associated normal toric variety,
and $\Nn$ the normalization of $\Na$. To begin with, we will see that $\Nn$ has a natural structure of toric variety
and so it is defined by some fan. Our first result shows that this fan can be identified with the Gr\"obner fan of
the ideal $J_n=\langle x^{a_1}-1,\ldots,x^{a_s}-1\rangle^{n+1}\subset\Su=k[\sd\cap\Z^d]$ (Theorem \ref{t. Nash = GF}).
This will be done essentially by comparing the action of the torus on the distinguished point of the dense orbit
of $\Nn$ and the induced action on the ideal $J_n$. By taking suitable limits, the same action will give us
the distinguished points of orbits in $\Nn$ and initial ideals of $J_n$.
\\
\\
The idea of comparing the fan defining $\Nn$ with a Gr\"obner fan is inspired by a similar idea that
appears in another paper of Yasuda in which the author defines a variant of $\Na$ in positive characteristic.
In the case of toric varieties, the author proves, using similar arguments, that this variant is determined by
a Gr\"obner fan (\cite{Y1}, Proposition 3.5). We also mention that a much more explicit combinatorial description
of the usual Nash blowup of toric varieties has been recently given by P. Gonz\'alez and B. Teissier in \cite{GT}
and by D. Grigoriev and P. Milman in \cite{GM}.
\\
\\
Later in the paper, we will study an analogue of the following well-known theorem of A. Nobile (\cite{No}): In characteristic zero,
the Nash blowup of a variety $X$ is an isomorphism if and only if $X$ is non-singular.
One can naturally ask if this theorem also holds for the higher Nash blowup. We answer this question affirmatively
when $X$ is a normal toric variety (Corollary \ref{c. Nobile without normalization}).
Using the description of $\Nn$ in terms of a Gr\"obner fan, the problem can be reduced to showing that this fan is a
non-trivial subdivision of the cone, say $\sigma$, defining $X$. By general results on the Gr\"obner fan, this is equivalent to
showing that there exists an element of some reduced Gr\"obner basis with the property that its initial part with respect to some
$w\in\sigma$ changes as we vary $w$ in $\sigma$.
\\
\\
The paper is organized as follows. In the first section we recall the basics of Gr\"obner bases but in a
slightly more general setting: instead of a polynomial ring we will consider a monomial subalgebra of the
polynomial ring. We give the definition of a Gr\"obner basis and Gr\"obner fan in this context and prove
some of their basic properties. Then, in the second and third section we prove, respectively, the description of $\Nn$
for normal toric varieties in terms of a Gr\"obner fan and the analogue of Nobile's theorem for normal toric varieties.

%

\section{Gr\"{o}bner fan of ideals in monomial subalgebras}\label{s. Groebner fan}

In this section we want to consider an intrinsic theory of Gr\"{o}bner bases of ideals in monomial
subalgebras of the polynomial ring. It can be verified that the basic theory of Gr\"{o}bner bases
(up to the existence and uniqueness of a reduced Gr\"obner basis) as shown, for example, in \cite{AL}
or \cite{CLO}, can be translated word by word to this setting. This is mainly because a completely
analogous division algorithm on any monomial subalgebra of the polynomial ring can be defined.


\subsection{Gr\"obner bases on $\Su$}

Let $\Su\subset\Po$ denotes the subalgebra generated by the monomials
$x^{a_i}:=x_1^{a_{i,1}}\cdot\ldots\cdot x_d^{a_{i,d}}$, where $a_i=(a_{i,1},\ldots,a_{i,d})\in\N^d$, and
$k$ is a field. Let $A:=\Z_{\geq0}(a_1,\ldots, a_s)=\{\sum_i\lambda_i a_i|\lambda_i\in\Z_{\geq0}\}$ denote
the semigroup generated by the $a_i's$. Define a monomial order on $\Su$ in the usual way. For instance,
any monomial order on $\Po$ restricts to a monomial order on $\Su$. However, the converse is not true.

\begin{exam}
Consider the subalgebra $k[x,xy]\subset k[x,y]$. Let $w=(\sqrt 3,-1)$. Define a monomial
order $\succ$ on the monomials of $k[x,xy]$ as follows:
$$x^ay^b\succ x^cy^d \Longleftrightarrow w\cdot(a,b)>w\cdot(c,d).$$
Suppose $\succ$ extends to a monomial order $\succ'$ on $k[x,y]$. Since, by definition, every monomial
$x^ay^b\in k[x,y]$ must satisfy $x^ay^b\succ'1$, then, in particular, we must have
$y\succ'1$. But then $x\cdot y\succ'1\cdot x=x$. Since $xy$ and $x$ are monomials on $k[x,xy]$ we
should have $xy\succ x$, which is clearly not true. Therefore, the monomial order $\succ$ cannot be
extended to $k[x,y]$.
\end{exam}

Let $>$ be a monomial order on $\Su$, $f=\sum_{i=1}^{r}\lambda_{\beta_i}x^{\beta_i}$ be a nonzero
polynomial in $\Su$, where $\beta_1>\beta_2>\cdots>\beta_r$. Define $lm(f):=x^{\beta_1}$, the leading
monomial of $f$; $lc(f):=\lambda_{\beta_1}$, the leading coefficient of $f$; $lt(f):=\lambda_{\beta_1}
\cdot x^{\beta_1}$, the initial form or leading term of $f$. In addition, we define $lm(0)=lc(0)=lt(0)=0$.
Finally, for $S\subset\Su$, we define the initial ideal of $S$, denoted $in_>(S)$, to be the ideal
generated (in $\Su$) by the leading terms of elements of $S$ with respect to $>$.

\begin{defi}
Fix a monomial order. A set of non-zero polynomials $G=\{g_1,\ldots,g_t\}$ contained in an ideal $I\subset\Su$,
is called a \textit{Gr\"{o}bner basis} for $I$ if for each $f\in I\setminus\{0\}$, there exists $i\in\{1,\ldots,t\}$
such that $lm(g_i)$ divides $lm(f)$ in $\Su$.
\end{defi}

\begin{defi}
A Gr\"{o}bner basis $G=\{g_1,\ldots,g_t\}$ is called \textit{reduced} if $lc(g_i)=1$ for all $i$, and no non-zero
monomial of $g_i$ is divisible by any $lt(g_j)$ for $j\neq i$.
\end{defi}

\begin{teo}
Fix a monomial order. Then every non-zero ideal $I$ has a unique reduced Gr\"{o}bner basis with respect
to this monomial order.
\end{teo}
\begin{proof}
As we said before, this can be proved in exactly the same way as in the polynomial ring case
(cf. \cite{AL}, Chapter 1).
\end{proof}


\subsection{Gr\"{o}bner fan}

The Gr\"obner fan of an ideal in $\Po$ is a subdivision of $\R^d_{\geq0}$ (see \cite{M}, Ch. 2, Def. 2.4.10).
Since we want to deal with monomial subalgebras, we will need to consider subdivisions of a little more general
cone in $\R^d$. In this section we give the analogous definition of Gr\"obner fan of an ideal in $\Su$. As
before, to prove that this Gr\"obner fan is indeed a fan, we can reproduce, word by word, the proof of the
polynomial ring case.
\\
\\
Let $\sd:=\R_{\geq0}(a_1,\ldots,a_s)\subset\R^d_{\geq0}$ be the cone generated by $a_1,\ldots,a_s$, and let
$\sigma\subset\R^d$ be its dual cone. Consider $w\in\sigma$, and $f=\sum c_ux^u\in\Su$. Define the initial form
$in_w(f)$ as the sum of terms $c_ux^u$ in $f$ with $w\cdot u$ maximized. The initial ideal of $I$ with respect to $w$
is defined as $in_w(I):=\langle in_w(f)|f\in I\rangle.$

\begin{pro}\label{p. C[w] cone-1}
Let $I$ be an ideal in $\Su$, let $w\in\sigma$ and consider
$$C[w]:=\{w'\in\sigma|in_w(I)=in_{w'}(I)\}.$$
Then $C[w]$ is the relative interior of a polyhedral cone inside $\sigma$.
\end{pro}
\begin{proof}
As in the polynomial ring case, it can be checked that
\begin{align}
C[w]=\{w'\in\sigma|in_{w'}(g)=in_w(g),\mbox{ for all }g\in G\},
\end{align}
where $G=\{g_1,\ldots,g_r\}$ is the reduced Gr\"{o}bner basis of $I$ with respect to $>_w$. Here $>$
is any monomial order on $\Su$ and $>_w$ is defined as $x^u>_wx^v$ if $u\cdot w>v\cdot w$ or
$u\cdot w=v\cdot w$ and $u>v$. For $g_i\in G$,
write $g_i=\sum_j c_{ij}x^{a_{ij}}+\sum_j c_{ij}'x^{b_{ij}}$, where $in_w(g_i)=\sum_j c_{ij}x^{a_{ij}}$.
The proposition then follows because the right-hand side set of $(1)$ equals
\begin{align}
\{w'\in\sigma|w'\cdot a_{ij}=w'\cdot a_{ik},w'\cdot a_{ij}>w'\cdot b_{ik}\mbox{ for }i=1,\ldots,r,
\mbox{ and all } j,k\}.\notag
\end{align}
This is the relative interior of a polyhedral cone by definition. See \cite{St}, Ch. 2, Prop. 2.3, or
\cite{M}, Ch. 2, Prop. 2.4.6 for details.
\end{proof}

\begin{pro}
Let $\overline{C[w]}$ be the closure of $C[w]$ in $\R^d$. Then the set $GF(I):=\{\overline{C[w]}|w\in\sigma\}$
forms a polyhedral fan.
\end{pro}
\begin{proof}
See \cite{St}, Ch. 2, Prop. 2.4, or \cite{M}, Ch. 2, Prop. 2.4.9.
\end{proof}

\begin{defi}
The set $GF(I)$ is called the \textit{Gr\"{o}bner fan} of $I$.
\end{defi}

In order to compute examples of Gr\"obner bases of ideals in $\Su$, we use the so-called
\textit{extrinsic algorithm for computing intrinsic Gr\"obner bases} (see \cite{St}, Algorithm 11.24).

\begin{exam}
Let $J=\langle xy+x,x^3y^3+x^2y^3\rangle\subset k[x,xy,x^2y^3]$. Let $>$ be the lexicographic order.
Let $w=(1,1)$. Implementing the extrinsic algorithm in $\mathtt{SINGULAR}$ $\mathtt{3}$-$\mathtt{1}$-$\mathtt{6}$,
we obtain the following reduced Gr\"obner basis with respect to $>_{w}$ (the leading terms are listed first):
$\{xy+x,x^4+x^3,x^2y^3-x^3\}$. Therefore (see prop. \ref{p. C[w] cone-1}), $C[(1,1)]=\{(p,q)\in\sigma|q>0, p>0, 3q>p\}.$
Similarly,
\begin{align}
C&[(4,1)]=\{(p,q)\in\sigma|q>0, p>0, p>3q, 2p+3q>0\},\notag\\
C&[(2,-1)]=\{(p,q)\in\sigma|0>q, 2p+3q>0, p>0\}.\notag
\end{align}
The resulting fan is shown in figure \ref{f. Groebner fan}.
\end{exam}

\begin{figure}[ht]
\begin{center}
\includegraphics[scale=0.7]{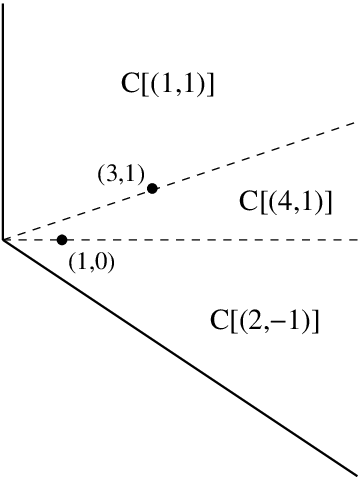}
\caption{Gr\"{o}bner fan of $J$.\label{f. Groebner fan}}
\end{center}
\end{figure}

%

\section{Higher Nash blowup of toric varieties}

In this section we introduce the notion of higher Nash blowup, defined by Takehiko
Yasuda in \cite{Y}. We state some basic properties and related results. Then we prove
that for normal toric varieties the higher Nash blowup is determined by the Gr\"obner
fan of a certain ideal.
\\
\\
The definition of the usual Nash blowup goes as follows (see \cite{No}):

\begin{defi}
Let $X\subset\C^m$ be an algebraic variety of pure dimension $d$. Consider the Gauss map:
\begin{align}
G:X\setminus\Si(X&)\rightarrow G(d,m)\notag\\
x&\mapsto T_xX,\notag
\end{align}
where $G(d,m)$ is the Grassmanian parameterizing the d-dimensional vector spaces in $\C^m$, and $T_xX$ is the
direction of the tangent space to $X$ at $x$. Denote by $X^*$ the Zariski closure of the graph of $G$. Call $\nu$
the restriction to $X^*$ of the projection of $X\times G(d,m)$ to $X$. The pair $(X^*,\nu)$ is called the \textit{Nash
blowup} of $X$.
\end{defi}

We can directly generalize this definition as follows (see \cite{OZ}, Section 1). Let us consider an
irreducible algebraic variety $X\subset\C^m$. Let $R$ be the ring of regular functions of $X$.
Consider the ideal $I=\ker(R\otimes R\rightarrow R)$, where $r\otimes r'\mapsto rr'$. We see $I$ as an
$R-$module via the map $R\rightarrow R\otimes R$, $r\mapsto r\otimes1$. For any $x\in X$, let
$(R_x,\m_x)$ be the localization of $R$ in $x$. Consider the following $\C\cong R_x/\m_x-$vector space:
$$T^n_xX:=(I_x/I_x^{n+1}\otimes\C)^{\vee}.$$
This is a vector space of dimension $N=\binom{d+n}{d}-1$ whenever $x$ is a non-singular point. Since $X\subset\C^m$
we have that $T^n_xX\subset T^n_x\C^m\cong\C^M$ where $M=\binom{m+n}{m}-1$, that is, we can see $T^n_xX$ as an
element of the grassmanian $G(N,M)$. Now consider the Gauss map:
$$G_n:X\setminus\Si(X)\rightarrow G(N,M)\mbox{ }\mbox{ }\mbox{ }x\mapsto T^n_xX.$$
Denote by $X_n$ the Zariski closure of the graph of $G_n$. Call $\nu$ the restriction to $X_n$ of
the projection of $X\times G(N,M)$ to $X$. The pair $(X_n,\nu)$ is called the \textit{Nash blowup of $X$ relative
to $I/I^{n+1}$} (this is a special case of a more general construction appearing in \cite{OZ}).
Viewed like this, it is clear that for $n=1$ this is exactly the usual Nash blowup of $X$ (in this case,
$T^1_xX=T_xX$, according to \cite{H}, Proposition 8.7).
\\
\\
This notion of Nash blowup of $X$ relative to $I/I^{n+1}$ is equivalent to the definition of higher Nash
blowup given by Yasuda (\cite{Y}, Proposition 1.8). The main difference between these constructions is that
Yasuda replaces the Grassmanian by a different parameter space of the variety: the Hilbert scheme of points.


\subsection{Higher Nash blowup}

Let $X:=\Sp R$, where $R=k[y_1,\ldots,y_s]/\mathfrak{p}$, $\mathfrak{p}$ is a prime ideal, and $k$ is an algebraically closed field
of characteristic zero. Consider $x\in X$ a $k-$point and let $\m$ be its corresponding maximal ideal in $R$.
Let $d=\dim X$. Let $x^{(n)}:=\Sp(R/\m^{n+1})$ be the $nth$ infinitesimal neighborhood of $x$.
If $X$ is smooth at $x$, then $x^{(n)}$ is a closed subscheme of $X$ of length
$N=\binom{d+n}{d}$ (i.e., $R/\m^{n+1}$ has length $N$ as an $R-$module). Therefore, it corresponds to a point
$$[x^{(n)}]\in Hilb_N(X),$$
where $Hilb_N(X)$ is the Hilbert scheme of $N$ points of $X$ (see \cite{Na}, Definition 1.2). If $X_{sm}$
denotes the smooth locus of $X$, then we have a map
$$\delta_n:X_{sm}\rightarrow Hilb_N(X),\mbox{ }\mbox{ }\mbox{ }x\mapsto [x^{(n)}].$$

\begin{defi}
(\cite{Y}, Definition 1.2) We define the \textit{nth Nash blowup} of $X$, denoted by $\Na$, to be the closure of the
graph of $\delta_n$ with reduced scheme structure in $X\times_k Hilb_N(X)$. By restricting the projection
$X\times_k Hilb_N(X)\rightarrow X$ we obtain a map $$\pi_n:\Na\rightarrow X.$$
This map is projective, birational, and it is an isomorphism over $X_{sm}$. In addition, $Nash_1(X)$ is
canonically isomorphic to the classical Nash blowup of $X$ (see \cite{Y}, Section 1).
\end{defi}

In \cite{Y}, the author conjectures that for $n$ big enough, the $nth$ Nash blowup of $X$ is non-singular
(see \cite{Y}, Conjecture 0.2). If the conjecture is true, this construction would give a one-step resolution
of singularities. In the same paper, the author proves that the conjecture is true for curves:

\begin{teo}
(\cite{Y}, Corollary 3.7) Let $X$ be a variety of dimension 1. Then for $n$ large enough, $\Na$ is non-singular.
\end{teo}

For varieties of higher dimension the answer remains unknown, even though Yasuda has stated
that the $A_3$-singularity is probably a counterexample to his conjecture (see \cite{Y1}, Remark 1.5).


\subsection{Normalization of the higher Nash blowup of a normal toric variety}

Let $\sigma\subset\R^d$ be a strictly convex rational polyhedral cone of dimension $d$.
Let $k[A]:=k[\sd\cap\Z^d]=\Su$. After suitable change of coordinates, we can assume that
$\Su\subset\Po$. Let $X:=\Sp k[A]$ be the corresponding $d$-dimensional normal toric
variety with torus $\T\subset X$. Since the torus $\T$ is dense in $X$ we first remark that
$$\Na=\overline{\{(x,\delta_n(x))|x\in X_{sm}\}}=\overline{\{(x,\delta_n(x))|x\in\T\}}.$$
In addition, $\T\cong\pi_n^{-1}(\T)=\{(x,\delta_n(x))|x\in\T\}=\{(x,[x^{(n)}])|x\in\T\}$, i.e.,
$\Na$ contains an open set isomorphic to the torus $\T$. The action of $\T$ on $X$ induces the
following action of $\T$ on $\Na$ (cf. \cite{Y}, Section 2.2):
$$\T\times\Na\rightarrow\Na,\mbox{ }\mbox{ }\mbox{ }(t,(x,[Z]))\mapsto(t\cdot x,[t\cdot Z]).$$
Over points $x\in\T$, i.e., $(x,[x^{(n)}])\in\pi_n^{-1}(\T)$, this action looks like:
$$\T\times\pi_n^{-1}(\T)\rightarrow\pi_n^{-1}(\T),\mbox{ }\mbox{ }\mbox{ }(t,(x,[x^{(n)}]))
\mapsto(t\cdot x,[(t\cdot x)^{(n)}]).$$
Since this action extends the action of the torus $\T\cong\pi_n^{-1}(\T)$ over itself, we obtain:

\begin{pro}\label{p. Nashn is toric}
Let $X$ be an affine normal toric variety. Then, for all $n\in\N$, $\Na$ is a toric variety
with dense torus $\pi_n^{-1}(\T)\cong\T$.
\end{pro}

This proposition is our starting point. Now we want to give a description of $\Na$ in terms of
fans and cones. In \cite{GT}, Section 7, the authors give a combinatorial description of non-normal
toric varieties having a finite open cover by $\T-$invariant affine sets. Unfortunately, we do not
know if such a cover exists for the higher Nash blowup of a toric variety.
Therefore we consider its normalization:
$$\eta:\Nn\rightarrow\Na.$$
Let $U:=\eta^{-1}(\T)$, which is dense since $\Nn$ is irreducible. Moreover, since $\T$ is contained
in the normal locus of $\Na$, we have that $U$ is isomorphic to $\T$. The action of $\T$ on $\Na$
induces the following action of $\T$ on $U$:
$$\T\times U\rightarrow U,\mbox{ }\mbox{ }\mbox{ }(t,\eta^{-1}(x,[x^{(n)}]))\mapsto\eta^{-1}(t\cdot x,
[(t\cdot x)^{(n)}]).$$
Since this action commutes with the normalization map restricted to $U$ then, by \cite{Se}, Lemma 6.1,
there is a unique action of $\T$ on $\Nn$ extending the action on $U$ and such that it commutes with
$\eta$. This implies that $\Nn$ is a (normal) toric variety with torus $U\cong\T$.
\\
\\
Now, since $\Nn$ is a normal toric variety, there exists a fan $\Sigma\subset N\otimes\R$, where $N$ is a lattice
of rank $d$, such that its associated normal toric variety is isomorphic to $\Nn$. The composition
$\pi_n\circ\eta:\Nn\rightarrow X$ is a morphism of toric varieties that sends the torus $U\subset\Nn$ to the
torus $\T\subset X$ in such a way that this restriction is a homomorphism of groups. Thus it is a toric
morphism. By \cite{O}, Theorem 1.13, there exists a morphism of lattices $\phi:N\rightarrow\Z^d$
compatible with $\Sigma$ and $\sigma$, and such that the induced morphism on the toric varieties is
$\pi_n\circ\eta$. On the other hand, since the normalization map is proper and birational we have that the
composition $\pi_n\circ\eta$ is a proper birational map of normal toric varieties. This implies that $\phi$ is
an isomorphism and $\sigma=\cup_{\tau\in\Sigma}\phi_{\R}(\tau)$, where $\phi_{\R}:N\otimes\R\rightarrow\Z^d\otimes\R$
is the tensor of $\phi$ and $\R$ (see \cite{O}, Chapter 1, Corollary 1.17). Because of this, we can assume that $N=\Z^d$,
$\phi$ is the identity, and $\Sigma$ is a refinement of $\sigma$.
\\
\\
Let $\mathbf{1}=(1,\ldots,1)$ be the distinguished point of the dense torus $\T\hookrightarrow X$ (see
\cite{CLS}, Chapter 3, Section 2, for the definition of distinguished point and its basic properties). Since
$\T\cong\pi_n^{-1}(\T)\cong\eta^{-1}(\pi_n^{-1}(\T))$, and since the action of $\T$ on $\Nn$
is induced by that of $\T$ on $X$, we have that $\eta^{-1}((\mathbf{1},\mathbf{1}^{(n)}))$ is the
distinguished point of the dense torus $\eta^{-1}(\pi_n^{-1}(\T))\subset\Nn$.
\\
\\
Consider $w\in\sigma$, and $f(x^{a_1},\ldots,x^{a_s})\in k[A]$. Since $k[A]\subset\Po$, we have
$f(x^{a_1},\ldots,x^{a_s})=F(x_1,\ldots,x_d)=\sum c_ux^u\in\Po$ (we are using different notation
for the same polynomial just to distinguish in which ring we are seeing the polynomial).
Let $d_w(f):=\max\{w\cdot u|c_u\neq0\}$.
Define
$$f_t:=t^{d_w(f)}F(t^{-w_1}x_1,\ldots,t^{-w_d}x_d)=
t^{d_w(f)}f(t^{-w\cdot a_1}x^{a_1},\ldots,t^{-w\cdot a_s}x^{a_s}).$$
Then we have $f_t=in_w(f)+t\cdot f'$, for some $f'\in k[A][t]$. Let $I_t:=\langle f_t|f\in I\rangle$ be
the ideal in $k[A][t]$ generated by the $f_t$.
\\
\\
Even though it may happen that $\langle (f_1)_t,\ldots,(f_r)_t\rangle\subsetneqq \langle f_1,\ldots,f_r\rangle_t$
in $k[A][t]$ (see \cite{E}, Exercise 15.25), we have the following result.

\begin{pro}\label{p. In-t = It-n}
Consider the ideal $\langle f_1,\ldots,f_r\rangle^{n+1}\subset k[A]$.
Let $w\in\sigma$. Then we have the following equality of ideals in $k[A][t,t^{-1}]$:
$$(\langle f_1,\ldots,f_r\rangle^{n+1})_t=\langle (f_1)_t,\ldots,(f_r)_t\rangle^{n+1}.$$
\end{pro}
\begin{proof}
The proof relies in the following facts that hold in $k[A][t,t^{-1}]$ and that can be checked by
a direct computation. For $f$, $g\in k[A]$, and $I$, $J$ ideals in $k[A]$:
\begin{itemize}
\item[(i)]$(f\cdot g)_t=f_t\cdot g_t.$
\item[(ii)]$(f+g)_t=\frac{f_t}{t^{d_w(f)-d_w(f+g)}}+\frac{g_t}{t^{d_w(g)-d_w(f+g)}}.$
\item[(iii)]$(I\cdot J)_t=I_t\cdot J_t.$
\end{itemize}
Now we proceed to prove the proposition, by induction. Let $n=0$. By definition,
$\langle (f_1)_t,\ldots,(f_r)_t\rangle\subset(\langle f_1,\ldots,f_r\rangle)_t$.
Now let $f=\sum_{i=1}^{r}h_if_i$. Then, by using (i) and (ii) we obtain
$f_t=\sum_{i=1}^{r}\frac{(h_if_i)_t}{t^{d_w(h_if_i)-d_w(f)}}
=\sum_{i=1}^{r}\frac{(h_i)_t}{t^{d_w(h_if_i)-d_w(f)}}\cdot(f_i)_t.$
This implies the other inclusion. Finally,
\begin{align}
(\langle f_1,\ldots,f_r\rangle^{n+1})_t&=
(\langle f_1,\ldots,f_r\rangle^n\cdot\langle f_1,\ldots,f_r\rangle)_t\notag\\
&=(\langle f_1,\ldots,f_r\rangle^n)_t\cdot(\langle f_1,\ldots,f_r\rangle)_t\notag\\
&=\langle (f_1)_t,\ldots,(f_r)_t\rangle^n\cdot
\langle (f_1)_t,\ldots,(f_r)_t\rangle\notag\\
&=\langle (f_1)_t,\ldots,(f_r)_t\rangle^{n+1}.\notag
\end{align}
\end{proof}
Let $J_0:=\langle x^{a_1}-1,\ldots,x^{a_s}-1\rangle\subset k[A]$ be the maximal ideal corresponding
to the closed point $\mathbf{1}\in X$ and let $J_n:=\langle x^{a_1}-1,\ldots,x^{a_s}-1\rangle^{n+1}$.
For $w\in\sigma$, consider the one-parameter subgroup $\lambda_w:k^*\to(k^*)^d$,
$t\mapsto t^w=(t^{w_1},\ldots,t^{w_d})$. Recall that $\lambda_w(t)$ acts on $\mathbf{1}$ in
the following way: $\lambda_w(t)\cdot\mathbf{1}=(t^{w\cdot a_1},\ldots,t^{w\cdot a_s})$.
Applying proposition \ref{p. In-t = It-n}, we have the following equality of ideals
in $k[A][t,t^{-1}]$: $(J_n)_t=\langle x^{a_1}-t^{w\cdot a_1},\ldots,x^{a_s}-t^{w\cdot a_s}\rangle^{n+1}$.
Thus, for any $t\in k^*$,
\begin{align}
\lambda_w(t)\cdot(\mathbf{1},\mathbf{1}^{(n)})&=(\lambda_w(t)\cdot\mathbf{1},(\lambda_w(t)\cdot\mathbf{1})^{(n)})
\notag\\
&=\Big((t^{w\cdot a_1},\ldots,t^{w\cdot a_s}),\Sp\frac{k[A]}{\langle x^{a_1}-t^{w\cdot a_1},\ldots,x^{a_s}-t^{w\cdot a_s}\rangle^{n+1}}\Big)\notag\\
&=\Big((t^{w\cdot a_1},\ldots,
t^{w\cdot a_s}),\Sp\frac{k[A]}{(J_n)_t}\Big).\notag
\end{align}
According to \cite{E}, Theorem 15.17 (this theorem is stated for the polynomial ring but
the same proof works if we replace polynomial ring by monomial subalgebra), we obtain:
\begin{equation}\label{eq1}
\lim_{t\rightarrow0}(\lambda_w(t)\cdot(\mathbf{1},\mathbf{1}^{(n)}))=
\Big(\lim_{t\rightarrow0}(\lambda_w(t)\cdot\mathbf{1}),\Sp\frac{k[A]}{in_w(J_n)}\Big).
\end{equation}

\begin{rem}
The notation we use for the limits of one-parameter subgroups is not standard. Usually the limit is denoted
just as $\lim_{t\rightarrow0}\lambda_w(t)$. Since we will be taking these limits at different levels ($X$,
$\Na$, and $\Nn$) we need to modify the standard notation in order to distinguish in which toric variety
we are working on.
\end{rem}

The following proposition shows that the fan defining $\Nn$ is a refinement of the
Gr\"{o}bner fan of $J_n$. In fact, we will see later that these two fans are actually equal.

\begin{pro}\label{p. Sigma divides GF}
Let $X=\Sp k[A]$ be the normal toric variety associated to the cone $\sigma$. Let $\Sigma$ be the fan
associated to the normalization of $Nash_n(X)$ and let $GF(J_n)$ be the Gr\"{o}bner fan of $J_n$. Then
$\Sigma$ is a refinement of $GF(J_n)$. In particular, there exists a surjective morphism of normal toric varieties
$$\xymatrix{\Nn\ar[r]^{\phi}&X_{GF(J_n)}}.$$
\end{pro}
\begin{proof}
To begin with, recall that the support of both $\Sigma$ and $GF(J_n)$ is $\sigma$. Let $w$ be in the relative
interior of $\sigma_1$, where $\sigma_1$ is a cone of $\Sigma$ different from $\{0\}$. Then there exists a
unique cone $\sigma_2$ of $GF(J_n)$ such that $w$ belongs to the relative interior of $\sigma_2$. Denote by
$\gamma_{\sigma_1}$ the distinguished point of $\sigma_1$ in $\Nn$. Now let $w'\neq w$ be in the relative
interior of $\sigma_1$. By \cite{CLS}, Proposition 3.2.2, we have
\begin{equation}\label{eq2}
\lim_{t\rightarrow0}(\lambda_w(t)\cdot\eta^{-1}((\mathbf{1},\mathbf{1}^{(n)})))=\gamma_{\sigma_1}=
\lim_{t\rightarrow0}(\lambda_{w'}(t)\cdot\eta^{-1}((\mathbf{1},\mathbf{1}^{(n)}))).
\end{equation}
By definition, $\lambda_w(t)\cdot\eta^{-1}((\mathbf{1},\mathbf{1}^{(n)}))=
\eta^{-1}((\lambda_w(t)\cdot\mathbf{1},(\lambda_w(t)\cdot\mathbf{1})^{(n)}))$. But now, since $\eta$ is a
continuous map,
\begin{align}
\eta(\lim_{t\rightarrow0}(\lambda_w(t)\cdot\eta^{-1}((\mathbf{1},\mathbf{1}^{(n)}))))&=
\lim_{t\rightarrow0}(\eta(\lambda_w(t)\cdot\eta^{-1}((\mathbf{1},\mathbf{1}^{(n)}))))\notag\\
&=\lim_{t\rightarrow0}(\eta(\eta^{-1}((\lambda_w(t)\cdot\mathbf{1},(\lambda_w(t)\cdot\mathbf{1})^{(n)}))))\notag\\
&=\lim_{t\rightarrow0}(\lambda_w(t)\cdot\mathbf{1},(\lambda_w(t)\cdot\mathbf{1})^{(n)})\notag\\
&=\lim_{t\rightarrow0}(\lambda_w(t)\cdot(\mathbf{1},\mathbf{1}^{(n)})).\notag
\end{align}
Similarly, $\eta(\lim_{t\rightarrow0}(\lambda_{w'}(t)\cdot\eta^{-1}((\mathbf{1},\mathbf{1}^{(n)}))))=
\lim_{t\rightarrow0}(\lambda_{w'}(t)\cdot(\mathbf{1},\mathbf{1}^{(n)}))$. Thus, by (\ref{eq1}) and (\ref{eq2}),
$\Sp k[A]/in_w(J_n)=\Sp k[A]/in_{w'}(J_n)$. This is an equality of closed subschemes of $\Sp k[A]$ according
to \cite{E}, Theorem 15.17. This implies that $in_{w}(J_n)=in_{w'}(J_n)$, i.e., $w'$ belongs to the relative
interior of $\sigma_2$. Therefore $\sigma_1\subset\sigma_2$. Since $\Sigma$ and $GF(J_n)$ have the same
support, we conclude that $\Sigma$ is a refinement of $GF(J_n)$.
\end{proof}

\begin{rem}
Notice that in the previous proof we cannot give a similar argument to show that $GF(J_n)$ is a refinement of $\Sigma$
since the normalization map may fail to be 1-1 over the non-normal locus. More precisely, let $\{r_i\}$,
$\{s_i\}$ be two sequences in $\pi_n^{-1}(\T)\subset\Na$ such that $\lim r_i=l=\lim s_i$, where
$l\in\Na\setminus\pi_n^{-1}(\T)$. Then it may happen that $\lim\eta^{-1}(r_i)\neq\lim\eta^{-1}(s_i)$.
\end{rem}

To prove that $\Sigma=GF(J_n)$ we will use the following lemma whose proof is similar to the one
of the previous proposition.

\begin{lem}\label{l. new lemma}
Suppose that $\sigma_1$, $\sigma_2$ are two cones in $\Sigma$ such that their relative interiors
are contained in the relative interior of some cone $\tau\in GF(J_n)$. Suppose that
the relative interior of $\sigma_1\cap\sigma_2$ is also contained in the relative interior of $\tau$.
Then
$$\eta(\gamma_{\sigma_1})=\eta(\gamma_{\sigma_2})=\eta(\gamma_{\sigma_1\cap\sigma_2}).$$
\end{lem}
\begin{proof}
Let $w$, $w'$, $w"$ be in the relative interior of $\sigma_1$, $\sigma_2$, and $\sigma_1\cap\sigma_2$,
respectively.
By \cite{CLS}, Proposition 3.2.2, we have:
\begin{align}
\lim_{t\rightarrow0}\lambda_{w}(t)\cdot\eta^{-1}((\mathbf{1},\mathbf{1}^{(n)}))&=\gamma_{\sigma_1},\notag\\
\lim_{t\rightarrow0}\lambda_{w'}(t)\cdot\eta^{-1}((\mathbf{1},\mathbf{1}^{(n)}))&=\gamma_{\sigma_2},\notag\\
\lim_{t\rightarrow0}\lambda_{w"}(t)\cdot\eta^{-1}((\mathbf{1},\mathbf{1}^{(n)}))&=\gamma_{\sigma_1\cap\sigma_2}.\notag
\end{align}
As in the previous proposition, using the fact that $\eta$ is a continuous map and (\ref{eq1}), we obtain:
\begin{align}
\eta(\gamma_{\sigma_1})&=\Big(\lim_{t\rightarrow0}(\lambda_{w}(t)\cdot\mathbf{1}),\Sp\frac{k[A]}{in_{w}(J_n)}\Big),
\notag\\
\eta(\gamma_{\sigma_2})&=\Big(\lim_{t\rightarrow0}(\lambda_{w'}(t)\cdot\mathbf{1}),\Sp\frac{k[A]}{in_{w'}(J_n)}\Big),
\notag\\
\eta(\gamma_{\sigma_1\cap\sigma_2})
&=\Big(\lim_{t\rightarrow0}(\lambda_{w"}(t)\cdot\mathbf{1}),\Sp\frac{k[A]}{in_{w"}(J_n)}\Big).\notag
\end{align}
On the other hand, $w$, $w'$, and $w"$ are also contained in the relative interior of $\tau$.
This implies, by definition of Gr\"{o}bner fan, that $in_{w}(J_n)=in_{w'}(J_n)=in_{w"}(J_n)$.
Consequently, $\eta(\gamma_{\sigma_1})=\eta(\gamma_{\sigma_2})=\eta(\gamma_{\sigma_1\cap\sigma_2})$.
\end{proof}

\begin{teo}\label{t. Nash = GF}
Let $X=\Sp k[A]$ be the normal toric variety associated to the cone $\sigma$. Let $\Sigma$ be the fan
associated to the normalization of $Nash_n(X)$ and let $GF(J_n)$ be the Gr\"{o}bner fan of $J_n$. Then
$\Sigma=GF(J_n)$.
\end{teo}
\begin{proof}
Suppose that $\Sigma\neq GF(J_n)$. According to proposition \ref{p. Sigma divides GF}, there exists a
$d-$dimensional cone $\tau\in\GF$ which is
subdivided in $\Sigma$. Let $\sigma_1,\sigma_2\in\Sigma$ be two different $d-$dimensional cones such that
their relative interiors, as well as the relative interior of $\sigma_1\cap\sigma_2$, are contained in the
relative interior of $\tau$. Notice that, in particular, $1\leq\mbox{dimension of }\sigma_1\cap\sigma_2<d$.
For $t\in\T$, we get from lemma \ref{l. new lemma} that
$$\eta(t\cdot\gamma_{\sigma_1\cap\sigma_2})=\eta(t)\cdot\eta(\gamma_{\sigma_1\cap\sigma_2})=
\eta(t)\cdot\eta(\gamma_{\sigma_1})=\eta(t\cdot\gamma_{\sigma_1})=\eta(\gamma_{\sigma_1}),$$
since $\eta$ is equivariant and $\gamma_{\sigma_1}$ is fixed for the torus action (because $\sigma_1$ is
of maximal dimension). This shows that $\eta(O_{\sigma_1\cap\sigma_2})=\eta(\gamma_{\sigma_1})$. Since
the dimension of the orbit $O_{\sigma_1\cap\sigma_2}$ is greater than zero, this contradicts the
fact that $\eta$ is a finite morphism. Therefore $\Sigma=GF(J_n)$.
\end{proof}

Let us illustrate the previous theorem by an example. Let
$\sigma$ be the cone generated by $(0,1)$ and $(4,-3)$. Then
$\sd\cap\Z^2$ is generated by $\{(1,0),(1,1),(3,4)\}$. Let
$X=\Sp\C[\sd\cap\Z^2]=\Sp\C[x,y,z]/\langle xy-z^4\rangle$. It is well known
that $Nash_1(X)$ is the blowup of the ideal $\langle x,y,z^3\rangle$ (see \cite{No}, Theorem 1). According to \cite{GM}, Section 4.3,
the fan corresponding to $\overline{Nash_1(X)}$ is the fan in figure \ref{f. fan Nash-1 A-3 example}.
\\
\\
Let us compare this fan with the Gr\"{o}bner fan of the ideal $J_1=\langle u-1,u^3v^4-1,uv-1\rangle^2\subset\C[u,u^3v^4,uv]$.
Consider the following vectors: $w_1=(1,0)$, $w_2=(3,-2)$. Implementing the so-called \textit{extrinsic algorithm
for computing intrinsic Gr\"obner bases} (see \cite{St}, Algorithm 11.24) in $\mathtt{SINGULAR}$ $\mathtt{3}$-$\mathtt{1}$-$\mathtt{6}$,
we find that the reduced Gr\"{o}bner bases of $J_1$ with respect to $w_1$ and $w_2$ are, respectively,
$$\{u^2v^2-2uv+1, u^2v-u-uv+1, u^2-2u+1, u^3v^4+u-4uv+2\},$$
$$\{u^2v^2-2uv+1, u^4v^5-u^3v^4-uv+1, u^6v^8-2u^3v^4+1, u+u^3v^4-4uv+2\}.$$
As in the proof of proposition \ref{p. C[w] cone-1}, we obtain the following open cones:
$$C[w_1]=\{(a,b)\in\sigma|a+b>0, a>0, a+2b>0, 2a+3b>0,3a+4b>0\},$$
$$C[w_2]=\{(a,b)\in\sigma|a+b>0, 3a+4b>0, -a-2b>0, -b>0, a>0\}.$$
The closures of these cones give precisely the fan in figure \ref{f. fan Nash-1 A-3 example}.

\begin{figure}[ht]
\begin{center}
\includegraphics{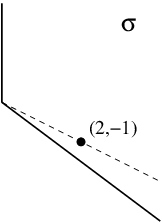}
\caption{Fan for $\overline{Nash_1(X)}$.\label{f. fan Nash-1 A-3 example}}
\end{center}
\end{figure}

%

\section{An analogue of Nobile's theorem}\label{s. Nobile}

In this final section we study an analogue of the following well-known theorem of A. Nobile
(\cite{No}, Theorem 2):

\begin{teo}
Let $X$ be an algebraic variety of pure dimension $d$ over an algebraically closed field of
characteristic zero. Let $(X^*,\nu)$ be the Nash blowup of $X$. Then, $\nu$ is an isomorphism
if and only if $X$ is non-singular.
\end{teo}

We will prove the analogue of this theorem in our particular context, that is, we consider only
normal toric varieties and Nash blowup is replaced by normalized higher Nash blowup. In view of the
results of the previous sections, we will be able to give a combinatorial proof using the theory of
Gr\"obner bases. Once this is done, it is an immediate consequence that the analogue of Nobile's theorem
for higher Nash blowup without normalization is also true for normal toric varieties (see Corollary
\ref{c. Nobile without normalization}).
\\
\\
Let $X$ be a normal toric variety. Let $(\Nn,\pi_n\circ\eta)$ be the $nth$ normalized higher Nash blowup of $X$.
One direction of the analogue of Nobile's theorem is clear; namely, if $X$ is non-singular then $\pi_n$ is an
isomorphism ($\pi_n$ only modifies singular points) and so is $\eta$. Therefore, if $X$ is non-singular,
$\pi_n\circ\eta$ is an isomorphism.
\\
\\
Let us suppose now that $X$ is singular. We want to prove that $\pi_n\circ\eta$ is not an isomorphism.
Let $\sigma\subset\R^d$ be a strictly convex polyhedral cone such that $X$ is the associated normal
toric variety. By theorem \ref{t. Nash = GF}, the fan corresponding to $\Nn$ is given by the Gr\"obner
fan of the ideal $J_n=\langle x^{a_1}-1,\ldots,x^{a_s}-1\rangle^{n+1}\subset\Su=k[\sd\cap\Z^d]$.
To prove that $\pi_n\circ\eta$ is not an isomorphism it suffices to prove that the Gr\"obner fan of $J_n$
truly subdivides $\sigma$. Indeed, suppose that $\GF$ is a non-trivial
subdivision of $\sigma$ and consider two cones $\sigma_1\neq\sigma_2$ in $\GF$, whose relative interiors
are contained in the relative interior of $\sigma$. Denote by $\gamma_{\sigma_1}$, $\gamma_{\sigma_2}$, and
$\gamma_{\sigma}$ the respective distinguished points. Then by \cite{CLS}, Lemma 3.3.21, we have:
$$(\pi_n\circ\eta)(\gamma_{\sigma_1})=\gamma_{\sigma}=(\pi_n\circ\eta)(\gamma_{\sigma_2}).$$
Since $\gamma_{\sigma_1}\neq\gamma_{\sigma_2}$, we see that $\pi_n\circ\eta$ is not injective, so it is not
an isomorphism.
\\
\\
Therefore, by definition of Gr\"obner fan, we need to find $w$, $w'\in\sigma$ such that $in_w(J_n)\neq in_{w'}(J_n)$.
As we saw in previous sections (see prop. \ref{p. C[w] cone-1}), this is equivalent to the following fact.
Fix some $w$ in the interior of $\sigma$ and let $>$ be any monomial order on $\Su$. Let $G$ be the reduced
Gr\"obner basis of $J_n$ with respect to $>_{w}$ (recall that $>_w$ is defined as $x^u>_wx^v$ if $u\cdot w>v\cdot w$
or $u\cdot w=v\cdot w$ and $u>v$). Then $in_w(J_n)\neq in_{w'}(J_n)$ for some $w'\in\sigma$ if and only if $in_w(g)\neq in_{w'}(g)$ for some $g\in G$.

\begin{rem}
We could formulate a similar question for ideals other than $J_n$, for $n\geq1$. Is it true that the fact
that the Gr\"obner fan of some ideal in $\Su$ does not subdivide $\sigma$ implies that $\sigma$ is regular?
The answer is no in general. Take for instance any monomial ideal. Any minimal monomial basis is already the
reduced Gr\"obner basis with respect to an order $>_w$ with $w\in\sigma$. The initial parts of these monomials
are trivially preserved when varying $w\in\sigma$. However, this does not imply regularity of $\sigma$.
But even for non-monomial ideals, something similar happens.
Consider the ideal $J_0$. Here the generators $\{x^{a_1}-1,\ldots,x^{a_s}-1\}$
form the reduced Gr\"obner basis of $J_0$ with respect to any $w\in\sigma$ and they also trivially satisfy the
conditions on the initial parts but this does not imply regularity of the cone $\sigma$.
\end{rem}

The strategy for the proof of the analogue of Nobile's theorem is to find an element of the reduced Gr\"obner basis
whose initial part changes as we vary $w\in\sigma$. To illustrate the method we consider the following family of
normal toric surfaces.

\begin{pro}
Let us consider the $A_m$-singularity, and let $k[x,x^my^{m+1},xy]$ be its ring of regular functions.
Let $J_n=\langle x-1, x^my^{m+1}-1,xy-1\rangle^{n+1}$. Then $\GF$ defines a non-trivial subdivision.
\end{pro}
\begin{proof}
Let $\sigma\subset\R^2$ be the cone generated by $(0,1)$ and $(m+1,-m)$. Denote by $R_1$ and $R_2$ the rays
generated by $(0,1)$ and $(m+1,-m)$, respectively. Fix some $w_0$ in the relative interior of $\sigma$ and
sufficiently close to $R_2$. Let $>$ be any monomial order on $k[x,x^my^{m+1},xy]$ and let $G$ be the reduced
Gr\"obner basis of $J_n$ with respect to $>_{w_0}$. We are going to show that there exists some $g\in G$ such
that its initial part changes as we vary $w\in\sigma$.
\\
\\
Since $(x-1)^{n+1}\in J_n$, there exists $g\in G$ such that $lt_{>_{w_0}}(g)|x^{n+1}$, i.e., $lt_{>_{w_0}}(g)=x^p$, $p\leq n+1$.
We consider two cases:
\begin{itemize}
\item[(1)] First suppose there is another monomial in $g$ different from a power of $x$. Since there are only
a finite number of monomials in $g$, then if $w_0$ is sufficiently close to $R_2$ we have that $in_{w_0}(g)=x^p$.
But now by taking $w$ sufficiently close to $R_1$, we have $in_{w}(g)\neq x^p$. This implies that $\GF$ defines a
non-trivial subdivision of $\sigma$.
\item[(2)] Now suppose that $g=x^p+\alpha_1 x^{p-1}+\cdots+\alpha_{p-1}x+\alpha_p$. Applying the division
algorithm to $(x-1)^{n+1}$ and $g$ we obtain $(x-1)^{n+1}=g\cdot q+r$, where $r=0$ or $r\neq0$ and $\deg_x r<\deg_x g$.
If $r\neq0$ then there is some $g'\in G$, $g'\neq g$ such that $lt_{>_{w_0}}(g')|lt_{>_{w_0}}(r)$ which implies that
$lt_{>_{w_0}}(g')|lt_{>_{w_0}}(g)$, contradicting the fact that $G$ is reduced. Therefore $r=0$ and so $g=(x-1)^p$.
Once again, we consider two cases:
\end{itemize}

\begin{itemize}
\item[(2.1)] Suppose $p<n+1$. In particular, $g=(x-1)^p\in J_n$ but this is impossible by lemma
\ref{l. p<n+1}, proved below.
\item[(2.2)] Suppose $p=n+1$. We are going to show that there is an element $h\in J_n$ such that
$lt_{>_{w_0}}(h)=x^n$, which again contradicts the fact that $G$ is reduced. We proceed by induction on $n$.
First we show that there is an element $h_1\in J_1$ such that $lt_{>_{w_0}}(h_1)=x$. Assume for the moment
that such an element exists.
Let $h_i:=(x-1)\cdot h_{i-1}\in J_i$, $i\geq2$. Then, by induction, $lt_{>_{w_0}}(h_i)=x^i$. Now we prove that such an $h_1$ exists.
Let $n=1$ and consider the following telescopic sums:
\begin{align}
x^{m+1}y^{m+1}&+1=(xy-1)\cdot\Big[\sum_{j=0}^m (xy)^{m-j}\Big]+2\notag\\
&=(xy-1)\cdot\Big[(xy-1)\cdot\Big(\sum_{j=1}^{m}j\cdot(xy)^{m-j}\Big)+(m+1)\Big]+2\notag\\
&=(xy-1)^2\cdot\Big(\sum_{j=1}^{m}j\cdot(xy)^{m-j}\Big)+(xy-1)\cdot(m+1)+2.\notag
\end{align}
This implies:
\begin{align}
x^{m+1}y^{m+1}-x^m&y^{m+1}-x+1=(xy-1)^2\cdot\Big(\sum_{j=1}^{m}j\cdot(xy)^{m-j}\Big)\notag\\
&-x^my^{m+1}-x+(m+1)\cdot xy-(m+1)+2.\notag
\end{align}
The term on the left equals $(x-1)\cdot(x^my^{m+1}-1)\in J_1$. Since $(xy-1)^2$ is also in $J_1$ we have that
$h_1:=x^my^{m+1}+x-(m+1)\cdot xy+(m+1)-2\in J_1$. If $w_0$ is sufficiently close to $R_2$, then $in_{w_0}(h_1)=x$ and so
$lt_{>_{w_0}}(h_1)=x$, as desired.
\end{itemize}
Therefore, by $(2.1)$ and $(2.2)$, case $(2)$ is impossible. By case $(1)$ we are done.
\end{proof}

\begin{rem}
Notice that the previous proof is also valid for any normal toric surface, since, according to \cite{O},
Proposition 1.21, there is an identical relation to that of $x, xy, x^my^{m+1}$, among any three
consecutive generators in the minimal generating set of the semigroup associated to the toric surface.
\end{rem}

Now we move into the general case. As before, let $\sigma\subset\R^d$ be a strictly convex rational polyhedral
cone of dimension $d$ and such that $\sd\subset\R^d_{\geq0}$. Let $\{a_1,\ldots,a_s\}\subset\Z^d_{\geq0}$ be the
minimal set of generators of $\sd\cap\Z^d$. We need two preliminary lemmas.
\\
\\
According to \cite{CLS}, Proposition 1.2.23, the set $\{a_1,\ldots,a_s\}$, contains the ray generators of the edges
of $\sd$ which we denote, after renumbering if necessary, by $\{a_1,\ldots,a_r\}$, as well as possibly some points
in the relative interior of $\{\sum_{i=1}^{r}\lambda_ia_i|0\leq\lambda_i<1\}$. Since $\sd$ has dimension $d$, we
must have $r\geq d$. Let us assume that $\sigma$ is not a regular cone.

\begin{lem}\label{l. h in J_n}
There exist $h\in J_n$ and some $w$ in the relative interior of $\sigma$ such that $lt_{>_w}(h)=(x^{a_i})^n$,
for some $i\in\{1,\ldots,r\}$.
\end{lem}
\begin{proof}
We proceed by induction on $n$. We are going to show that there exist $h_1\in J_1$ and some $w\in\sigma$ such
that $lt_{>_w}(h_1)=x^{a_i}$ for some $i\in\{1,\ldots,r\}$. Assume for the moment that such $h_1$ and $w$ exist.
Let $h_l=(x^{a_i}-1)\cdot h_{l-1}\in J_l$, $l\geq2$.
Then, by induction, $lt_{>_w}(h_l)=(x^{a_i})^l$. Now we prove that such $h_1$ and $w$ exist. Let $n=1$ and
consider the following map of $k$-algebras:
$$\phi:k[y_1,\ldots,y_s]\rightarrow\Su,\mbox{ }\mbox{ }\mbox{ }y_i\mapsto x^{a_i}.$$
Let $\overline{J_1}:=\langle y_1-1,\ldots,y_s-1 \rangle^2+\ker\phi$. Since $\sigma$ is not a regular cone, we must
have $s>d$. Consider a subset of $\{a_1,\ldots,a_r\}$ consisting of $d$ linearly independent elements (such a subset
exists since $\sd$ has dimension $d$). After renumbering, if necessary, we may assume that this subset is $\{a_1,\ldots,a_d\}$.
Let $A$ be the transpose of the matrix whose rows are $a_1,\ldots,a_d$, in this order. Let $\lambda':=(\lambda_1',\ldots,\lambda_d')$
be the solution of the equation $A z=a_{d+1}$, i.e., $\lambda'=A^{-1}a_{d+1}$. The entries of $A$ are all integers
as well as those of $a_{d+1}$, whence $\lambda'\in\Q^d$. By multiplying by suitable integers and after renumbering,
if necessary, we obtain the following relation:
$$\lambda_1a_1+\cdots+\lambda_ta_t=\lambda_{t+1}a_{t+1}+\cdots+\lambda_{d+1}a_{d+1},$$
where $\lambda_i\in\Z_{\geq0}$ for all $i$, and for some $t\in\{1,\ldots,d\}$. This implies that
$y_1^{\lambda_1}\cdots y_t^{\lambda_t}-y_{t+1}^{\lambda_{t+1}}\cdots y_{d+1}^{\lambda_{d+1}}\in\ker\phi.$
\\
\\
Consider the change of coordinates $y_i\mapsto y_i'+1$. Then
$$(y_1'+1)^{\lambda_1}\cdots (y_t'+1)^{\lambda_t}-(y_{t+1}'+1)^{\lambda_{t+1}}\cdots (y_{d+1}'+1)^{\lambda_{d+1}}$$
belongs to $K$, where $K$ is the image of $\ker\phi$ under the change of coordinates, and consequently, it also belongs
to $\langle y_1',\ldots,y_s'\rangle^2+K$.
Since $\langle y_1',\ldots,y_s'\rangle^2$ contains all monomials of degree two in the variables $y_i'$, the polynomial
$\delta_1y_1'+\cdots+\delta_{d+1}y_{d+1}'$ is also in $\langle y_1',\ldots,y_s'\rangle^2+K$, for some non-zero
coefficients $\delta_i$. Undoing the change of coordinates, we obtain $\tilde{h}:=\delta_1y_1+\cdots+\delta_{d+1}y_{d+1}+c\in\overline{J_1}$,
where $c$ is a constant. Hence $h_1:=\phi(\tilde{h})=\delta_1x^{a_1}+\cdots+\delta_{d+1}x^{a_{d+1}}+c\in J_1$. Now consider two cases (recall
that $r$ denotes the number of edges of $\sd$):
\begin{itemize}
\item[(1)]If $r>d$ then $a_{d+1}\in\{a_1,\ldots,a_r\}$. Consequently, $lt_{>_w}(h)=x^{a_i}$, for
some $i\in\{1,\ldots,r\}$ and any $w\in\sigma$, as desired.
\item[(2)]Suppose that  $r=d$ and recall that $\{a_1,\ldots,a_s\}$ is the minimal set of generators of
$\sd\cap\Z^d$. In particular,
$a_{d+1}=\sum_{i=1}^{d}\lambda_ia_i$, where $0\leq\lambda_i<1$.
Denote by $H$ the hyperplane generated by $\{a_1,\ldots,a_{d-1}\}$. Then $H\cap\sd$ is a facet of $\sd$, i.e.,
there exists $w\in\sigma$ such that $w^{\perp}=H$. In particular, $w\cdot a_i=0$ for $i=1,\ldots,d-1$, and $w\cdot a_d>0$.
If $a_{d+1}\in H$ then $lt_{>_w}(h)=x^{a_d}$, as desired. Otherwise, $w\cdot a_{d+1}>0$. Now we choose $w'$
sufficiently close to $w$ in the relative interior of $\sigma$ and such that $0<w'\cdot a_i<w'\cdot a_d$ and $0<w'\cdot a_i<w'\cdot a_{d+1}$
for all $i=1,\ldots,d-1$. We know that $a_{d+1}=\sum_{i=1}^{d}\lambda_ia_i$, where, in particular, $0<\lambda_d<1$.
This fact allow us to choose $w'$ satisfying also $w'\cdot a_{d+1}<w'\cdot a_d$. Therefore, $lt_{>_{w'}}(h)=x^{a_d}$.
This concludes the proof of the lemma.
\end{itemize}
\end{proof}

\begin{lem}\label{l. p<n+1}
If $p<n+1$, then $(x^{a_i}-1)^p\notin J_n$, for every $i$.
\end{lem}
\begin{proof}
For convenience of notation, we take $i=1$ and we assume that $a_{11}>0$. Let
$f_t:=(x^{a_1}-1)^{t_1}\cdot(x^{a_2}-1)^{t_2}\cdots(x^{a_s}-1)^{t_s},$ where $\sum_j t_j=n+1$. Suppose that
\begin{align}
(x^{a_1}-1)^p=\sum h_tf_t,\label{e. p<n+1}
\end{align}
for some $h_t\in\Su$. We will get a contradiction by taking derivatives with respect to $x_1$. When we take the
first derivative with respect to $x_1$ of $\sum h_tf_t$, every summand $h_tf_t$ produces two summands, according
to Leibniz' rule of derivation. Each of these new summands contains a factor $(x^{a_1}-1)^{r_1}\cdot(x^{a_2}-1)^{r_2}\cdots(x^{a_s}-1)^{r_s}$,
where $n\leq\sum_jr_j\leq n+1$. Continuing this way, after differentiating $p$ times with respect to $x_1$, every summand
in the resulting sum contains a factor $(x^{a_1}-1)^{r_1}\cdot(x^{a_2}-1)^{r_2}\cdots(x^{a_s}-1)^{r_s}$, where $0<n+1-p\leq\sum_jr_j\leq n+1$.
\\
\\
On the other hand, the first derivative with respect to $x_1$ of $(x^{a_1}-1)^p$ is $(x^{a_1}-1)^{p-1}\cdot m$,
where $m$ is some monomial. The second derivative will produce two summands, each one being a product of
$(x^{a_1}-1)^r$ where $p-1\leq r\leq p$, and some monomial. Continuing this way, after $p-1$ derivations, the
resulting sum consists of summands of the form $(x^{a_1}-1)^r\cdot m$, $1\leq r\leq p$, and where there is
exactly one summand such that $r=1$. The next derivation produces a non-zero monomial plus summands of the
form $(x^{a_1}-1)^r\cdot m$, where $1\leq r\leq p$.
\\
\\
Therefore, after derivating each side of equation (\ref{e. p<n+1}) $p$ times, and evaluating the resulting
polynomials in $(1,1,\ldots,1)$ we obtain zero on the right hand and something different from zero on the left
hand. This is a contradiction.
\end{proof}

Now we are ready to prove the analogue of Nobile's theorem in our context.

\begin{teo}\label{t. analogue Nobile}
Let $X$ be the normal toric variety defined by $\sigma$. Let $\pi_n\circ\eta:\Nn\rightarrow X$ be the normalized
higher Nash blowup of $X$. Then if $X$ is singular, $\pi_n\circ\eta$ is not an isomorphism.
\end{teo}
\begin{proof}
Let $w\in\sigma$ be as in lemma \ref{l. h in J_n}. Let $G$ be the reduced Gr\"obner basis of $J_n$ with respect
to $>_w$, where $>$ is any monomial order on $\Su$. By definition, $(x^{a_i}-1)^{n+1}\in J_n$. For each $i$,
there exists $g_i\in G$ such that $lt_{>_w}(g_i)|(x^{a_i})^{n+1}$. For $i\in\{1,\ldots,r\}$, this implies that
$lt_{>_w}(g_i)=(x^{a_i})^{p_i}$, where $p_i\leq n+1$. Now we consider two cases:
\begin{itemize}
\item[(1)]Suppose there is some $i\in\{1,\ldots,r\}$ such that $g_i$ contains some monomial $x^\delta$ that is
not a power of $x^{a_i}$. By definition of $>_w$, $(x^{a_i})^{p_i}$ is a monomial of $in_w(g_i)$. On the other hand,
since $p_i a_i$ is in the ray generated by $a_i$ (which is a ray of the cone $\sd$), there exists $w'\in\sigma$ such
that $w'\cdot(p_i a_i)=0$ and $w'\cdot\delta>0$. Now we choose $w''$ sufficiently close to
$w'$ in the relative interior of $\sigma$ and such that $0<w''\cdot(p_i a_i)<w''\cdot\delta$. This implies
that $(x^{a_i})^{p_i}$ is not a monomial of $in_{w''}(g_i)$. Consequently, $C[w]\neq C[w'']$ and so the
Gr\"obner fan of $J_n$ is not trivial. Here $C[w]$ denotes the equivalence class of $w$ in the Gr\"obner fan
of $J_n$.
\item[(2)]Suppose that
$g_i=(x^{a_i})^{p_i}+\alpha_{i,1}(x^{a_i})^{p_i-1}+\cdots+\alpha_{i,p_i-1}(x^{a_i})+\alpha_{i,p_i}$,
where $i\in\{1,\ldots,r\}$. Applying the division algorithm in one variable we obtain:
$$(x^{a_i}-1)^{n+1}=g_i\cdot q_i+r_i,$$ where $r_i=0$ or $r_i\neq0$ and $\deg_{x^{a_i}}(r_i)<\deg_{x^{a_i}}(g_i)$.
If $r_i\neq0$ for some $i$, the previous equality implies $r_i\in J_n$, and so there exists $g\in G$, $g\neq g_i$
for all $i$, such that $lt_{>_w}(g)|lt_{>_w}(r_i)$. But this implies that $lt_{>_w}(g)|lt_{>_w}(g_i)$, which contradicts
the fact that $G$ is reduced. Therefore $r_i=0$ for all $i$, implying $g_i=(x^{a_i}-1)^{p_i}$,
where $p_i\leq n+1$. By lemma \ref{l. p<n+1}, $p_i$ cannot be smaller than $n+1$, i.e., $p_i=n+1$ for all $i$.
According to lemma \ref{l. h in J_n}, there exists $h\in J_n$ such that $lt_{>_w}(h)=(x^{a_i})^n$. Once again,
this gives a contradiction.
\end{itemize}
By $(1)$ and $(2)$, the Gr\"obner fan of $J_n$ defines a non-trivial subdivision and so $\pi_n\circ\eta$ is not an isomorphism.
\end{proof}

As an immediate consequence, the analogue of Nobile's theorem for higher Nash blowup without normalization is
also true for normal toric varieties.

\begin{coro}\label{c. Nobile without normalization}
Let  $X$ be a normal toric variety and let $(\Na,\pi_n)$ be its $nth$ higher Nash blowup. Then $\pi_n$ is an isomorphism
if and only if $X$ is non-singular.
\end{coro}
\begin{proof}
Suppose $\pi_n$ is an isomorphism. In particular, $\Na$ is normal whence $\Nn\simeq\Na$. By the previous theorem,
this implies that $X$ is non-singular.
\end{proof}

\section*{Acknowledgements}
I want to thank Mark Spivakovsky for his encouragement and for having detected some mistakes
in previous versions of this manuscript. I also want to thank Takehiko Yasuda for having
kindly answered several questions I had on the higher Nash blowup. Some ideas presented
here started from discussions with him. Thanks also to Stepan Orevkov, Nikolay Vasiliev,
Dmitry Pavlov, and Marcello Bernardara for stimulating discussions. Finally, I want to
thank the referee for his careful reading and for his suggestion that simplified my original
proof of Theorem \ref{t. Nash = GF}.



\begin{thebibliography}{XXX}
\addcontentsline{toc}{chapter}{\numberline{Bibliograf'ia}}
\bibitem[AL]{AL}W. Adams, P. Loustaunau; \textit{An Introduction to Gr\"{o}bner Bases}. Graduate
Studies in Mathematics, \textbf{3}, AMS, Providence, 1994.
\bibitem[CLO]{CLO}D. Cox, J. Little, D. O'Shea; \textit{Ideals, Varieties, and Algorithms}, Third Edition,
Undergraduate Texts in Mathematics, Springer, New York, 2007.
\bibitem[CLS]{CLS}D. Cox, J. Little, H. Schenck; \textit{Toric varieties}, Graduate Studies in Mathematics,
Volume 124, AMS, 2011.
\bibitem[DGPS]{DGPS}W. Decker, G.-M. Greuel, G. Pfister, H. Sch{\"o}nemann:
\newblock {\sc Singular} {3-1-6} --- {A} computer algebra system for polynomial computations.
\newblock {http://www.singular.uni-kl.de} (2012).
\bibitem[E]{E}D. Eisenbud; \textit{Commutative Algebra with a View Toward Algebraic Geometry}, Graduate
Texts in Mathematics, Vol. 150, Springer-Verlag, New York, 1995.
\bibitem[GT]{GT}P. D. Gonz\'alez, B. Teissier; \textit{Toric Geometry and the Semple-Nash modification}, Revista de
la Real Academia de Ciencias Exactas, F\'isicas y Naturales, Serie A Matem\'aticas, DOI 10.1007/s13398-012-0096-0, (2012).
\bibitem[GM]{GM}D. Grigoriev, P. Milman; \textit{Nash desingularization for binomial varieties as Euclidean
division, a priori termination bound, polynomial complexity in dim 2}, Adv. Math. 231 (2012), no. 6, 3389-3428.
\bibitem[H]{H}R. Hartshorne; \textit{Algebraic Geometry}, Graduate Texts in Mathematics, Vol. 52, 1977.
\bibitem[M]{M}D. Maclagan, R. Thomas; \textit{Computational Algebra and Combinatorics of Toric Ideals},
Proceedings of the International Conferences organized by Bhaskaracharya Pratishthana, Pune and Harish-Chandra
Research Institute, Allahabad, during 8 - 13 December, 2003, Ramanujan Mathematical Society Lecture Notes Series,
Vol. 4.
\bibitem[Na]{Na}H. Nakajima; \textit{Lectures on Hilbert Schemes of Points on Surfaces}, University Lecture Series,
Vol. 18, American Mathematical Society, Providence, 1991.
\bibitem[No]{No}A. Nobile; \textit{Some properties of the Nash blowing-up}, Pacific Journal of Mathematics,
\textbf{60}, (1975), 297-305.
\bibitem[O]{O}T. Oda; \textit{Convex bodies and algebraic geometry}, Ergebnisse der Mathematik und ihrer
Grenzgebiete (3), Vol. 15, Springer-Verlag, Berlin, 1988.
\bibitem[OZ]{OZ}A. Oneto, E. Zatini; \textit{Remarks on Nash blowing-up}, Rend. Sem. Mat. Univ. Politec.
Torino \textbf{49} (1991), no. 1, 71-82, Commutative algebra and algebraic geometry, II (Italian) (Turin 1990).
\bibitem[Se]{Se}C. S. Seshadri; \textit{Quotient Spaces Modulo Reductive Algebraic Groups}, Annals of Mathematics,
Vol. 95, No. 3, (1972).
\bibitem[St]{St}B. Sturmfels; \textit{Gr\"{o}bner Bases and Convex Polytopes}, University Lecture Series, Vol. 8,
American Mathematical Society, Providence, 1996.
\bibitem[Y]{Y}T. Yasuda; \textit{Higher Nash blowups}, Compositio Math. \textbf{143} (2007), no. 6, 1493-1510.
\bibitem[Y1]{Y1}T. Yasuda.; \textit{Universal flattening of Frobenius}, American Journal of Mathematics,
Volume 134, No. 2, (2012), p. 349-378.
\end{thebibliography}
\end{document}